\documentclass[a4paper,12pt]{article}
\usepackage{hyperref}
\usepackage[utf8]{inputenc}
\usepackage{mathrsfs,amsfonts,amssymb,amsmath,array,times,mathptmx}
\newcommand\dmf[1]{\qquad\qquad(#1)}
\newcommand\unit{{\mathbf{1}}}
\newcommand\chfc[1]{1_{#1}}

\newcommand\sg[2]{\left(#2_{#1}\right)_{#1\ge0}}
\newcommand\gr[2]{\left(#2_{#1}\right)_{#1\in\mathbb{R}}}
\newcommand\sequ[2]{\left(#2_{#1}\right)_{#1\in\mathbb{N}}}
\renewcommand\d{{\mathrm{d}}}
\newtheorem{theorem}{Theorem}
\newtheorem{proposition}{Proposition}

\newtheorem{definition}{Definition}
\newtheorem{corollary}{Corollary}
\newtheorem{remark}{Remark}
\newtheorem{example}{Example}
\numberwithin{theorem}{section}
\numberwithin{proposition}{section}
\numberwithin{lemma}{section}
\numberwithin{definition}{section}
\numberwithin{corollary}{section}
\numberwithin{example}{section}
\newenvironment{proof}[1][]{\begin{trivlist}\item[\textit{Proof#1.}]}{\rule{0.3em}{2ex}\end{trivlist}}
\newcommand{\set}[1]{\left\{ #1\right\}}
\newcommand{\norm}[1]{\left\| #1\right\|}
\newcommand{\scpro}[2]{\langle #1,#2\rangle}
\newcommand{\map}[3]{ #1\colon #2\mapsto #3}
\newcommand\uvtimes{{\mathbin{\mathop{~\otimes~}\limits_{u,v}}}}
\newcommand{\Utimes}{\mathbin{\mathop{~\otimes~}\nolimits_{\mskip-12mu\U}}}
\newcommand\uvptimes{{\mathbin{\mathop{~\otimes~}\limits_{u',v'}}}}

\renewcommand\uvtimes{{{}_u{\otimes}_{v}}}
\renewcommand\uvptimes{{{}_{u'}{\otimes}_{v'}}}

\newcommand\NRp{{\mathbb{R}_{\ge0}}}
\newcommand\E{\mathcal{E}}
\newcommand\F{\mathcal{F}}
\newcommand\G{\mathcal{G}}
\newcommand\B{\mathcal{B}}
\newcommand\K{\mathcal{K}}
\newcommand\M{\mathcal{M}}
\newcommand\U{\mathcal{U}}

\newcommand\C{\mathcal{C}}
\makeatletter
\def\p@equation{{\upshape(\theequation)}\@eqtest}
\def\p@figure{{Fig. \thefigure}\@eqtest}
\def\p@table{{Tab. \thetable}\@eqtest}
\def\p@example{{Example \theexample}\@eqtest}
\def\p@theorem{{Theorem \thetheorem}\@eqtest}
\def\p@proposition{{{Proposition \theproposition}}\@eqtest}
\def\p@lemma{{Lemma  \thelemma}\@eqtest}
\def\p@definition{{Defn.  \thedefinition}\@eqtest}

\def\@eqtest#1{\ifx\csname#1\expandafter\expandafter\expandafter\@gobble
\expandafter\csname\fi}

\def\p@enumi{\labelenumi\expandafter\@gobble}
\def\p@enumii{\labelenumii\expandafter\@gobble}
\def\p@enumiii{\labelenumiii\expandafter\@gobble}
\def\p@enumiv{\labelenumiv\expandafter\@gobble}

\makeatother

\newcommand{\sB}{\mathcal{B}}

\begin{document}
\title{\Large\bfseries The Relation of  Spatial and  Tensor  Product of Arveson Systems --- The Random Set Point of View}
\author{Volkmar Liebscher}
\date{\day09\month09\year2014\relax\today,17:30:21
}

\maketitle
\begin{abstract}
  We characterise the embedding of the spatial product of two Arveson systems into their  tensor product using the random set technique. An important implication is that the spatial tensor product does not depend on the choice of the reference units, i.e. it  is an intrinsic construction.   There is a continuous range of examples coming from the zero sets of Bessel processes where the two products do not coincide. The lattice of all subsystems of the tensor product is analised in different cases. As a by-product, the Arveson systems coming from Bessel zeros prove to be primitive in the sense of \cite{JMP11a}. 
\end{abstract}
\section{Introduction}
In a series of seminal papers in 1989 and 1990, \textsc{Arveson}  associated with every \emph{$E_0$-semigroup} (a semigroup of unital endomorphisms) on $\sB(H)$ its continuous product system of Hilbert spaces, \emph{Arveson system} for short.  Briefly, it is a measurable family of separable Hilbert spaces $\E=\sg t\E$ with an associative identification
\begin{displaymath}
\E_s\otimes\E_t=\E_{s+t}, \qquad s,t\ge0.
\end{displaymath}
\textsc{Arveson}  showed in \cite{Arv89a} that $E_0$-semi\-groups are classified by their Arveson system up to cocycle conjugacy. By a \emph{spatial} Arveson system we understand a pair $(\E,u)$ of an Arveson system $\E$ and a \emph{normalised unit} $u$. The latter  is a measurable section $u=\sg t u$ of unit vectors $u_t\in\E_t$ that factor as
\begin{displaymath}
u_s\otimes u_t=u_{s+t}, \qquad s,t\ge0
\end{displaymath}
with additionally $\norm{u_t}=1$. For a thorough account on Arveson systems we refer to the monograph \cite{Arv03}.

It is known that the  structure of a spatial Arveson system $(\E,u)$ depends on the choice of the reference unit $\sg t u$. In fact, \textsc{Tsirelson} \cite{Tsi08} and \textsc{Markiewicz and Powers} \cite{MP08a}  showed that for an example Arveson system  $\sg t\E$ with  normalised units $\sg t u$ and $ \sg tv$ that there does not exist an automorphism of $\E$ that sends $\sg t u$ to $\sg t v$. Thus we have to  distinguish Arveson systems and spatial Arveson systems carefully. 

  The focus of the present paper is the spatial product of two spatial Arveson systems $(\E,u)$ and $(\F,v)$, which  is formally given by
 \begin{equation}\label{eq:formindlim1}
    (\E\uvtimes{}\F)_t=\mathop{\mathrm{lim}}\limits_{(t_1,\dots,t_n)\in \Pi^t}\bigotimes^{n}_{i=1}((u_t\otimes v_t^\perp)\oplus(\mathbb{C} u_t\otimes v_t)\oplus  (u_t^\perp\otimes v_t)).
  \end{equation}
Here, the limit is taken over finer and finer partitions of $[0,t]$. This is exactly the description of the product system arising from Powers sum of $E_0-$semigroups, see \cite{Ske03c,BLS08}. It also arises as a special case  of inclusion systems \cite{BM09a}. For this structure,  the two units $u$ and $v$ are glued together into one unit of the product. 

Interestingly, \cite{Ske06d} showed that a similar construction works for product systems of Hilbert \emph{modules}, too. This was very important, since  for general product systems of Hilbert modules, the fibrewise tensor product need not yield a product system. Unfortunately, the random set technique used below was not extended to the module situation yet. Thus, we deal here with Arveson systems only. 

Not spatial Arveson systems as such, but also their spatial product depends \textit{a priori} on the choice of the reference units of its factors. This immediately raises the question whether different choices of references units yield isomorphic products or not. In \cite{BLMS11} this question was answered in the affirmative sense. One aim of the present papers is to show how this universality comes quite naturally from the random set point of view on Arveson systems. Only after knowing the result from a former version of the present paper, \cite{BLMS11} achieved the same goal without explicit reference to random sets.

From \ref{eq:formindlim1} it is easy to see that the spatial product is a subsystem of the tensor product system. Nevertheless, the nature of this embedding  is not completely clarified.  Using the random set construction of \cite{Lie09a}, we  characterise here the embedding of the spatial product into the tensor product easily. This random set  structures arise naturally with any embedding  $\G\subseteq\E$ of Arveson systems in the following way. Consider the projections 
\begin{displaymath}
  \mathrm{P}_{s,t}=\unit_{\E_s}\otimes \mathrm{Pr}_{\G_{t-s}}\otimes\unit_{\E_{1-t}}\qquad\in\B(\E_1=\E_s\otimes \E_{t-s}\otimes\E_{1-t})
\end{displaymath}
on $\E_1$. The  fulfil the relation
\begin{displaymath}
   \mathrm{P}_{r,s}\mathrm{P}_{s,t}=\mathrm{P}_{r,t}\qquad0\le r\le s\le t\le1
\end{displaymath}
It was one of the results of \cite{Lie09a}, inspired by \cite{Tsi00},  to give the interesting part of the (normal and separable) representation theory of these relations, identifying the projection $\mathrm{P}_{s,t}$ with multiplication by the $\set{0,1}$-valued random  variable
\begin{displaymath}
  X_{s,t}(Z)=\left\{
    \begin{array}[c]{cl}
      1&Z\cap[s,t]=\emptyset\\
      0&Z\cap[s,t]\ne\emptyset
    \end{array}
\right.
\end{displaymath}
on the space $\mathcal{C}_{[0,1]}=\set{Z\subseteq[0,1]: Z\mbox{~closed} }$ equipped with a suitable probability  measure. Multiplicity is  encoded in a direct integral of Hilbert spaces as usual, see \ref{th:directintegralps} below.  Having the representation for such projections at hand, it is quite easy to compute functions of those projections. In the present situation, we want to compute the projection onto $(\E\uvtimes{}\F)_1$ which characterises $\E\uvtimes{}\F$ completely. A few basic facts about the relevant measures then yield independence of  the construction from the reference units, solving a question raised by Powers in \cite{Pow04a}. This solution was presented also in \cite{BLMS11} with a different proof not using the random set structure explicitly. But that proof, unobviously,  computed just consequences of the random set structure without reference to it.  We hope to convince the reader that using random sets gives a much more clear derivation of the results and that the present paper is worthwhile.  This result would have been  trivial, if   for any pair of normalised units there would exist an isomorphism of the  product systems  mapping one unit to the other. That property was named \emph{amenability} in \cite{Bha00}. But, since \cite{Tsi08,MP08a} we know that there are examples of product systems without this property  and  our result is nontrivial.

  Note that there are  examples that the two products form  nonisomorphic product systems, provided by  \cite{Pow04a} together with \cite{APP05}. Below, another series of examples is provided. Those examples  use  the Arveson  systems coming from the zero sets of Bessel diffusions as introduced already by \textsc{ Tsirelson} \cite{Tsi00}. Those examples are  all of type $\mathrm{II}_0$ but nonisomorphic. As a by-product, we show that those product systems are really primitive in the sense that they contain only trivial subsystems. Thus they are also prime product systems in the sense of \cite{JMP11a}.  Further, spatial products of the Bessel zero Arveson systems  have a quite similar structure, with a rich group of automorphisms, compared  to the behavior of type $\mathrm{I}_1$ Arveson systems under the (spatial) product.  Still, we do not know whether these examples really differ from those in  \cite{Pow04a}.

\paragraph{Acknowledgements} This work began with a RiP stay at the Oberwolfach institute in 2007. The main part  was a completed during a stay at the Universit\`a de Molisse, Campobasso, 2009. Special   thanks to M.Skeide  for the warm hospitality during the latter stay, and for  a lot of discussion and encouragement later as well. Further,  discussions with B.V.R. Bhat, K.Waldorf, D.Markiewicz and  P. Moerters helped a lot to shape this work.

 \section{Continuous product systems of Hilbert spaces}
Let us start with some definitions.
  \begin{definition}
An \emph{Arveson  system} is a measurable family $\E=(\E_t)_{t\ge0}$ of  separable Hilbert spaces  endowed with a measurable family of  
unitaries $V_{s,t}:\E_s\otimes \E_t\mapsto \E_{s+t}$ for all $s,t\ge0$ which fulfils  for all $r,s,t\ge0$
  \begin{displaymath}
    V_{r,s+t}\circ(\unit_{\E_r}\otimes V_{s,t})= V_{r+s,t}\circ( V_{r,s}\otimes\unit_{\E_t}).
  \end{displaymath}
  \end{definition}
  \begin{definition}
    A unit $ u$ of an Arveson  system is a measurable non-zero section $\sg t u$ through $\sg t\E$, which satisfies for all $s,t\ge0$
    \begin{displaymath}
       u_{s+t}= V_{s,t}u_s\otimes  u_t=u_s\otimes u_t.
    \end{displaymath}
If $ u$ is normalised ($\norm{ u_t}=1\forall t\ge0$), the pair $(\E, u)$ is also called \emph{spatial} Arveson system.
For any (spatial) Arveson system $\E$ denote $\U_1(\E)$ the set of all normalised units of $\E$. 
  \end{definition}
  \begin{remark}
   We do not make the definition of  measurability   more explicit throughout this paper. For a  thorough discussion see \cite{Lie09a}, especially section 7 there. Most importantly, 
  by  \cite[Theorem 7.7]{Lie09a} existence of a compatible  measurable structure for an Arveson system is determined by the algebraic structure (given by the family $(V_{s,t})_{0\le s\le t}$) alone. The example Arveson systems introduced below obey that condition.     

Another distinction to \cite{Arv89a} is the inclusion of the trivial  0- and 1-dimensional product systems and of time 0. This way  the order structure of Arveson subsystems becomes simpler.  

In the sequel we drop the operators $V_{s,t}$ whenever there is no loss of  precision. 
  \end{remark}

\begin{definition}
Let additionally $\F$ be another Arveson system with unitaries $(W_{s,t})_{0\le s,t}$.

\begin{enumerate}\item 
  We say that $\theta=(\theta_t)_{t\ge0}$ is an isomorphism of
  product systems if $\theta_t:\E_t\mapsto\F_t$ is a unitary for all $t\ge0$ and for all $s,t\ge0$
  \begin{displaymath}
    \theta_{s+t}\circ V_{s,t}=W_{s,t}\circ(\theta_s\otimes\theta_t).
  \end{displaymath}
If $\F=\E$, $\theta$ is called automorphism.
\item We call $\F$ a subsystem of $\E$ if $\F_t\subseteq\E_t$ for all $t\ge0$ and   $W_{s,t}=V_{s,t}\mathop{|}_{\F_s\otimes\F_t}$  for all $s,t\ge0$.
\end{enumerate}
  \end{definition}
Then $$\mathrm{Aut}(\E)=\set{\theta: \theta\mbox{ is an automorphism of $\E$}}$$ is a group under pointwise composition, called gauge group of $\E$.

According to \cite{Ske06d}, \cite{BLS08} we introduce now another product, the spatial  product of Arveson  systems. For this and further use later, observe by  \cite[Theorem 5.7]{Lie09a} that for an Arveson system $\mathcal {E}$ the set
\begin{displaymath}
  \mathscr{S}(\mathcal{E})=\set{\mathcal {F}:\mathcal{F} \mbox{ is an Arveson  subsystem of }\mathcal{E}}
\end{displaymath}
forms a (complete)  lattice with respect to the fibrewise inclusion order.  Thus  $\E'\vee\F'$ denotes the smallest Arveson subsystem containing both $\E'$ and $\F'$. Under slight abuse of notation, we identify normalised units $u$ with the subsystem $\sg t{\mathbb{C}u}$.

\begin{definition}
Let $(\E,u)$ and $(\F,v)$ be two spatial Arveson systems. We define their\emph{ spatial  product} as   
  \begin{displaymath}
    \E\uvtimes{}\F:=(u\otimes\F)\vee(\E\otimes v)\subseteq\E\otimes\F
  \end{displaymath}
\end{definition}

For a more explicit  definition (see e.g.\ \cite{BLS08}), let   
\begin{displaymath}
  \Pi^t=\set{(t_1,\dots,t_n):n\in \set{1,2,\dots}, t_i>0, t_1+\dots+t_n=t}
\end{displaymath}
denote the set of interval partitions of $[0,t]$ (in a suitable parametrisation).
We order $\Pi^t$ by $(t_1,\dots,t_n)\prec(s_1,\dots,s_m)$ if $n\le m$ and there is a strictly increasing map $\varphi:\set{1,\dots,n,n+1}\mapsto\set{1,\dots,m,m+1}$ with $\varphi(1)=1$, $\varphi(n+1)=m+1$, and
\begin{displaymath}
  t_i=s_{\varphi(i)}+\dots+s_{\varphi(i+1)-1}\forall i=1,\dots,n.
\end{displaymath}

Further, for any vector $w$ in a Hilbert space  denote $w^\perp$ its orthogonal complement. 
\def\citt{\cite[Proposition 2.7]{BLMS11}}
\begin{proposition}[\citt]
\label{prop:inductive limit}
Let $(\E, u)$ and $(\F, v)$ be two spatial Arveson systems. Define Hilbert spaces 
\begin{equation}
\label{eq:inclusionsystem}  G^{u,v}_t= u_t\otimes v_t^\perp\oplus\mathbb{C} u_t\otimes v_t\oplus  u_t^\perp\otimes v_t.
\end{equation}

Then for all $t>0$
\begin{equation}
\label{eq:limitinclusionsystem}
    (\E\uvtimes{}\F)_t=\mathop{\mathrm{lim}}\limits_{(t_1,\dots,t_n)\in \Pi^t}G^{u,v}_{t_1}\otimes G^{u,v}_{t_2}\otimes\dots\otimes G^{u,v}_{t_{n-1}}\otimes G^{u,v}_{t_n}.
  \end{equation}
\end{proposition}
\begin{remark}
The work on inclusion systems \cite{BM09a} is a direct generalisation of this inductive limit technique.   
\end{remark}


The   main question now is whether the  inclusion $\E\uvtimes\F\subseteq \E\otimes\F$ might be proper. 
The  answer is reported later.

For any (spatial) Arveson system we introduce its type \textrm{I} part 
\begin{displaymath}
  \E^\U=\bigvee_{u\in\U_1(\E)}u,
\end{displaymath}
the Arveson subsystem generated by its units.  $\E$ is called type \textrm{I}, if $\E=\E^\U$, type \textrm{II} if $\E^\U\ne 0,\E$,  and type \textrm{III} if $\E^\U=0$ or $\U_1(\E)=\emptyset$.  $\E^\U$ is isomorphic to an Arveson system $(\Gamma(L^2([0,t],\K)))_{t\ge0}$ of symmetric Fock spaces  for some separable Hilbert space $\K$ \cite{Arv89a}. $\dim \K$ is an invariant called \emph{index of $\E$}. We subclassify  the types $\mathrm{I},\mathrm{II}$ according to their index. This means, e.g., that for $n\in \mathbb{N}$ an Arveson system of type $\mathrm{I}_n$ is isomorphic to   $((\Gamma(L^2([0,t],\mathbb{C}^n))))_{t\ge0}$ \cite{Arv89a}.  It is easy to see that the index is additive under both the tensor product and the spatial product.   
  \section{Product Systems and Random Sets}
If $\E$ is an Arveson system, there is an important unitary one parameter group  $(\tau_t)_{t\in\mathbb{R}}\subset\B(\E_1)$ acting for $t\in(0,1)$ with regard to  the representations
$\E_{1-t}\otimes\E_t\cong\E_1\cong\E_t\otimes\E_{1-t}$ as \emph{flip}:
\begin{equation}
\label{eq:defflipgroup}
  \tau_t x_{1-t}\otimes x_t=x_t\otimes  x_{1-t}\dmf{x_{1-t}\in\E_{1-t},~ x_t\in\E_t}.
\end{equation}
The operators $\tau_t$ for $t\notin(0,1)$ are obtained by $1-$periodic continuation. These unitaries yield via $\Theta_t(a)=\tau_t^*a\tau_t$, $a\in\B(\E_1) $, a periodic one parameter  automorphism group $\gr t\Theta$ on $\mathcal{B}(\mathcal{E}_1)$.

Observe that any  Arveson  subsystems  $\G$ of an Arveson system $\E$ yields a family  $(\mathrm{P}^\G_{s,t})_{0\le s<t\le1}$ of projections   
\begin{equation}\label{eq:defPst}
  \mathrm{P}^\G_{s,t}=\unit_{\E_s}\otimes\mathrm{Pr}_{\G_{t-s}}\otimes\unit_{\E_{1-t}}\qquad\in\B(\E_1=\E_s\otimes \E_{t-s}\otimes\E_{1-t}).
\end{equation}
This family  fulfils the following relations
\begin{eqnarray*}
  \mathrm{P}^\G_{s,t}\mathrm{P}^\G_{t,u}&=&\mathrm{P}^\G_{s,u}\qquad0\le s\le t\le u\le1\\
\mathrm{P}^\G_{s+u,t+u}&=&\Theta_u(\mathrm{P}^\G_{s,t})\qquad0\le s\le t\le1,-s\le u\le 1-t.
\end{eqnarray*}

The following theorem makes the r\^ole of (distributions of) random sets in Arveson systems apparent. Thereby, let  $\mathcal{C}_{[0,1]}$ denote the space of closed subsets of the unit interval. It is a separable compact space itself, with a corresponding $\sigma$-field of Borel sets. We implicitly assume all probability measures on $\mathcal{C}_{[0,1]}$ to be defined on this $\sigma$-field. 
\def\citt{\cite[Theorem 3.16]{Lie09a}}
\begin{theorem}[\citt]\label{th:existencemuomega}
   Let $\E$ be an Arveson system,   $\omega$ be a faithful  normal state  on  $\B(\E_1)$ and $\G$ be an Arveson  subsystem of $\E$. 

Then there is a unique probability measure  $\mu_\omega$ on  $\mathcal{C}_{[0,1]}$ with  
  \begin{displaymath}
    \label{eq:mueta}
   \mu_\omega(\set{Z:Z\cap\left(\smash{\bigcup_i}[s_i,t_i]\right)=\emptyset})=\omega(\mathrm{P}^\G_{s_1,t_1}\cdots \mathrm{P}^\G_{s_k,t_k})\dmf{0\le s_i<t_i\le1}
  \end{displaymath}

Further, there is a unique    normal isomorphism  $j_\G$, 
\begin{displaymath}
  j_\G:L^\infty(\mu_\omega)\mapsto\set{\smash{\mathrm{P}^\G_{s,t}}:0\le s<t\le1}''\subset\B(\E_1),
\end{displaymath}
with
\begin{displaymath}
  j_\G(\chfc{\set{Z\cap[s,t]=\emptyset}})=\mathrm{P}^\G_{s,t}\dmf{0\le s<t\le1}.
\end{displaymath}
\end{theorem}
  \section{Stationary factorising   measure  types}
We saw above that the space $L^\infty(\mu_\omega)$ seems to play a more fundamental r\^ole than the measure $\mu_\omega$ itself. That means, equivalent measures yield the same structure. We want to formalise this. 

Recall that 
   a \emph{measure type} is  an equivalence class  of probability measures, where  equivalence of measures $\mu$ and $\nu$ (symbol $\mu\sim\nu$) means that    $\mu$ and $\nu$ have the same null sets.

On $\mathcal{C}_{[0,1]}$, we have the natural operations of restriction $Z\mapsto Z_{s,t}=Z\cap[s,t]$ and circular shift $Z\mapsto Z+t:=\overline{Z+t\pmod 1}$. The first gives rise to an image measure $\mu_{s,t}$, the second to the image measure $\mu+t$. The convolution  associated with $\cup$ is denoted by $\ast$. These notions transfer naturally to measure types.
\begin{definition}
A measure type $\M$     on  $\mathcal{C}_{[0,1]}$ is
   \emph{stationary factorising}
if 
\begin{eqnarray*}
  \M_{r,t}&=&\M_{r,s}\ast\M_{s,t}\dmf{0\le r<s<t\le 1}\\
  \M_{r,s}+t&=&\M_{r+t,s+t}\dmf{0\le r<s<s+t\le 1}
\end{eqnarray*}  
 \end{definition}
\def\citt{\cite[Theorem 3.22 and Corollary 6.2]{Lie09a}}
\begin{theorem}[\citt]
 In the situation of  \ref{th:existencemuomega},
 \begin{displaymath}
\M^\G=\set{\mu_\omega:\omega\mbox{~faithful~}}
\end{displaymath}
is a stationary factorising measure
type.
\end{theorem}

  \section{The embedding $\E\uvtimes\F\subseteq\E\otimes\F$}
 We use also the following extension of  \ref{th:existencemuomega}: \def\citt{\cite[Proposition 3.32]{Lie09a}}
  \begin{proposition}[\citt]
Suppose for two subsystems $\G_1,\G_2$ of an Arveson system that the projection families  $\mathrm{P}^{\G_1},\mathrm{P}^{\G_2}$ commute.

Then  there exists for all normal states $\omega$ on $\mathcal{B}(\E_1)$ a unique probability measure 
     $\mu_\omega$ on $\mathcal{C}_{[0,1]}\times\mathcal{C}_{[0,1]}$ with
    \begin{displaymath}
      \mu_\omega(\set{(Z_1,Z_2):Z_j\cap\smash{\bigcup_i}[s^j_i,t^j_i]=\emptyset})=\omega(\prod_j\prod_i\mathrm{P}^{\G_j}_{s^j_i,t^j_i}).
  \end{displaymath}
The corresponding measure type is denoted $\M^{\G_1,\G_2}$.

 Further, there exists unique  isomorphism $J_{\G_1,\G_2}:L^\infty(\M^{\G_1,\G_2})\mapsto\mathcal{B}(\E_1)$ with
  \begin{displaymath}
    J_{\G_1,\G_2}(\chfc{\set{(Z_1,Z_2):Z_j\cap[s,t]=\emptyset}})=P^j_{s,t}\dmf{j=1,2}.
  \end{displaymath}
\end{proposition}
  Denote for a closed set $Z\subseteq \NRp$ the set of its limit points  by ${\hat{Z}}$. I.e.,
  \begin{displaymath}
    {\hat{Z}}=\set{t\in Z:t\in \overline{Z\setminus\set t}}=\set{t\in Z:\exists Z\ni t_n\ne t, n\in \mathbb{N}, t=\lim_{n\to\infty} t_n}.
  \end{displaymath}
  This means that $Z\setminus {\hat{Z}}$ is the countable set of isolated points of $Z$. 
\def\citt{\cite[Proposition 3.33]{Lie09a}}
\begin{example}[\citt]
 Consider  $\G_1=\mathbb{C}u$ for a unit $\sg tu$ and  $\G_2=\E^{\mathscr{U}}$. Then 
\begin{displaymath}
    J_{\G_1,\G_2}(f)=J_{u,\E^\U}(f)=J_u(g)
  \end{displaymath}
where $g(Z)=f(Z,{\hat{Z}})$.
\end{example}

  \begin{proposition}\label{thm:inclusion}
For spatial Arveson systems $(\E,u)$, $(\F,v)$ it holds 
    \begin{equation}\label{eq:projectionspatialproduct}
      \mathrm{P}^{\E\uvtimes\F}_{s,t}=J_{\E\otimes v,u\otimes\F}(\chfc{\set{(Z_1,Z_2):Z_1\cap Z_2\cap[s,t]=\emptyset}})
    \end{equation}
  \end{proposition}
  \begin{remark}
    Compare this expression to \cite[Theorem 2.1]{Pow04a}, which seems to compute $ \mathrm{P}^{\G}_{0,1}$ in a special case. Observe that the latter projection identifies already the corresponding Arveson subsystem.  
  \end{remark}
  \begin{proof}
We use  \ref{prop:inductive limit}. Using the notation \ref{eq:inclusionsystem} we derive
\begin{displaymath}
  G^{u,v}_t\otimes \E_{1-t}\otimes\F_{1-t}= J_{\E\otimes v,u\otimes\F}(\chfc{\set{(Z_1,Z_2):Z_1\cap [0,t]=\emptyset\mbox{~or~}Z_2\cap [0,t]=\emptyset}}).
\end{displaymath}
By normality of $J_{\E\otimes v,u\otimes\F}$ we obtain
\begin{eqnarray*}
  \mathrm{P}^{\E\uvtimes\F}_{0,1}&=&\mathop{\mathrm{lim}}\limits_{(t_1,\dots,t_n)\in \Pi^1}\mathrm{Pr}_{G^{u,v}_{t_1}}\otimes\mathrm{Pr}_{ G^{u,v}_{t_2-t_1}}\otimes\dots\otimes \mathrm{Pr}_{G^{u,v}_{t_n-t_{n-1}}}\\
&=& \mathop{\mathrm{lim}}\limits_{(t_1,\dots,t_n)\in \Pi^1}J_{\E\otimes v,u\otimes\F}(\chfc{\set{(Z_1,Z_2): \forall i: Z_1\cap [\sum_{j=1}^it_j,\sum_{j=1}^{i+1}t_{j}]=\emptyset\mbox{~or~}Z_2\cap [\sum_{j=1}^it_j,\sum_{j=1}^{i+1}t_{j}]=\emptyset}})\\
&=&J_{\E\otimes v,u\otimes\F}( \mathop{\mathrm{lim}}\limits_{(t_1,\dots,t_n)\in \Pi^1}\chfc{\set{(Z_1,Z_2):\forall i: Z_1\cap [\sum_{j=1}^it_j,\sum_{j=1}^{i+1}t_{j}]=\emptyset\mbox{~or~}Z_2\cap [\sum_{j=1}^it_j,\sum_{j=1}^{i+1}t_{j}]=\emptyset}})\\
&=&J_{\E\otimes v,u\otimes\F}(\chfc{\set{(Z_1,Z_2):Z_1\cap Z_2=\emptyset}}).
\end{eqnarray*}
Formula \ref{eq:projectionspatialproduct} for $s\ne0$ or $t\ne 1$ follows immediately since  $\mathrm{P}^{\E\uvtimes\F}_{0,1}$ determines the whole Arveson system $\E\uvtimes\F$.
This completes the proof.
\end{proof}
  \begin{proposition}\relax\label{cor:alsoZhatdoesit}
The relation $\E\uvtimes\F=\E\otimes\F$ is valid  if and only if 
    \begin{equation}\label{eq:emptyintersection}
      Z_1\cap Z_2=\emptyset\dmf{\M^u \otimes\M^v-\mbox{a.s.}}
    \end{equation} if and only if \label{prop:alsoZhatdoesit} 
    \begin{equation}\label{eq:emptyintersectionb}
      \widehat{Z_1}\cap \widehat{Z_2}=\emptyset\dmf{\M^u \otimes\M^v-\mbox{a.s.}}
    \end{equation}
  \end{proposition}
  \begin{proof}
    The first assertion is clear. The second one follows from the fact that $Z_1\setminus \widehat{Z_1}$ and  $Z_2\setminus \widehat{Z_2}$  are countable. Since $(\E,u)$ and $(\F,v)$ are spatial, both $Z_1$and $Z_2$ are different from $[0,1]$ almost surely.  Then  we know from \cite[Proposition 4.4]{Lie09a} that such a stationary factorising random set almost never meets a countable set and we conclude
    \begin{displaymath}
Z_1\cap Z_2=\widehat{Z_1}\cap Z_2=\widehat{Z_1}\cap \widehat{Z_2}\dmf{\M^u \otimes\M^v-\mbox{a.s.} }
\end{displaymath}
This completes the proof. 
 \end{proof}
\begin{corollary}
    If the lattice $\mathscr{S}(\E\uvtimes\F)$ has finite depth and \ref{eq:emptyintersection} is not fulfilled, $\E\otimes\F\not\cong \E\uvtimes\F$.

  \end{corollary}
  \begin{proof}
If \ref{eq:emptyintersection} is not valid, $\E\uvtimes\F$ is a proper subsystem of $\E\otimes\F$. If both were isomorphic, iteration of this observation would yield an infinite chain of Arveson subsystems in  $\mathscr{S}(\E\uvtimes\F)$. 
  \end{proof}

  \begin{corollary}
In the following cases we have that $ \E\uvtimes\F= \E\otimes\F$:
    \begin{enumerate}
    \item One of $\E$ or $\F$ is type $\mathrm{I}$.
    \item $Z$ is countable $\M^{u}$-a.s. or $\M^v$-a.s.
   \end{enumerate}
  \end{corollary}
  \begin{proof}
    1. Suppose $\F$ is type $\mathrm{I}$. Then $\widehat{Z}=\emptyset$  $\M^v$-a.s., since $Z$ is  $\M^v$-a.s.\ finite by \cite[Proposition 3.33]{Lie09a}. \ref{eq:emptyintersectionb} gives the desired conclusion.

2. \cite[Proposition 4.4]{Lie09a} shows that for any countable $Z_2\in \mathcal{C}_{[0,1]}$ $Z\cap Z_2=\emptyset$ for  $\M^u$-a.a. $Z$. This yields again the conclusion.

  \end{proof}

\section{The spatial product does not depend on the units}

A direct consequence of \ref{cor:alsoZhatdoesit} is that $\E\uvtimes\F$ is intrinsic, i.e. it does not depend on the choice of $u$ and $v$. You can find  another a bit  more complicated  formulation of the proof without explicit reference to random sets in \cite[Theorem 3.1]{BLMS11}.    

  \begin{theorem}\label{thm1}
Let $(\E, u)$, $(\E, u')$, $(\F, v)$ and  $(\F, v')$ be  spatial Arveson systems.

Then
\begin{displaymath}
  \E\uvtimes\F=\E\uvptimes\F.
\end{displaymath}
  \end{theorem}
  \begin{proof}
We know from \cite[Proposition 3.33]{Lie09a} for $f\in L^\infty(\M^u)$ that $J_u(f\circ \widehat{\cdot})=J_{\E^\mathscr{U}}(f)$. By \ref{prop:alsoZhatdoesit} this shows
    \begin{eqnarray*}
      \mathrm{P}^{\E\uvtimes\F}_{s,t}&=&J_{\E\otimes v,u\otimes\F}(\chfc{\set{(Z_1,Z_2):\widehat{Z_1}\cap \widehat{Z_2}\cap[s,t]=\emptyset}})\\
&=&J_{\E\otimes \F^{\mathscr{U}},\E^{\mathscr{U}}\otimes\F}(\chfc{\set{(Z_1,Z_2): Z_1\cap Z_2\cap[s,t]=\emptyset}}).
    \end{eqnarray*}
The last expression  is independent of $u$ and $v$.
  \end{proof}

  \begin{corollary}
    It holds
    \begin{displaymath}
      \E\uvtimes\F=(\E^\U\otimes\F)\vee(\E\otimes\F^\U).
    \end{displaymath}
  \end{corollary}

Thus, we use $\E\Utimes\F$ as new symbol for $\E\uvtimes\F$.  This is also consistent with the amalgamation procedure from \cite{BM09a}.
Note that \cite{BLMS11} introduced the symbol $\E\otimes^0\F$. 
 \section{From measure types to   Hilbert spaces}
Before we study special examples of Arveson systems, we want to present the general mechanism for constructing those examples. It dates back  to \textsc{Tsirelson} \cite{Tsi00}.

If $\mu\sim\mu'$ are two measures on the same space (here $\mathcal{C}_{[0,1]}$), the abelian von Neumann algebras  $L^\infty(\mu)$ and $L^\infty(\mu')$ coincide, and we observe a canonical space $L^\infty(\M)$ if $\M$ is the measure type of $\mu$ and $\mu'$. Now we want to present an intrinsic  construction of a Hilbert space $L^2(\M)$. In this we follow   
 \cite{Tsi00,Tsi03} or originally  \cite{Acc76}.

 Define  for any $\mu,\mu'\in\M$ a unitary  ${U_{\mu,\mu'}}:{L^2(\mu)}\mapsto{L^2(\mu')}$  through 
\begin{equation}
  \label{eq:defUmumus}
U_{\mu,\mu'}\psi(Z)=\sqrt{\frac{\mathrm{d}\mu'}{\mathrm{d}\mu\mskip2mu}(Z)}\psi(Z)  \dmf{\psi\in L^2(\mu),\mu-\mbox{a.a.~}Z\in \mathcal{C}_{[0,1]}}.
\end{equation}
Then 
 \begin{equation}
\label{eq:defL2measuretype}
   L^2(\M)=\set{(\psi_\mu)_{\mu\in\M}: \psi_\mu\in L^2(\mu)\forall \mu\in\M,\psi_{\mu'}=U_{\mu,\mu'}\psi_\mu \forall \mu,\mu'\in\M }
 \end{equation}
is a Hilbert space with the inner product
\begin{displaymath}
  \scpro{\psi}{\psi'}_{L^2(\M)}=\int \overline{\psi_\mu}\psi'_\mu\mathrm{d}\mu.
\end{displaymath}
This inner product is independent from the choice of $\mu\in\M$. 

Now we obtain 
\def\citt{\cite[Proposition 4.3]{Lie09a}}
\begin{proposition}[\citt]
Let $\M$ be a stationary factorising measure type on $\mathcal{C}_{[0,1]}$. Define operators $V_{s,t}:L^2(\M_{0,s})\otimes L^2(\M_{0,t})\mapsto L^2(\M_{0,s+t})$ for $0\le s,t$, $s+t\le1$ through  
  \begin{displaymath}
   (V_{s,t} \psi\otimes  \psi')_{\mu_s\ast(\mu'_t+s)}(Z)=\psi_{\mu_s}(Z\cap[0,s])\psi'_{\mu'_t}(Z\cap[s,s+t]-s)
  \end{displaymath}
Then $V_{s,t}$ are well-defined unitaries and give rise to an Arveson   system $\E=\E^\M=\sg t\E$ with   $\E_t=L^2(\M_{0,t})$ for  $0\le t\le1$.

A unit $\sg tu$ of $\E$ is determined  by  
\begin{displaymath}
  (u_t)_{\mu_{0,t}}(Z)=\mu_{0,t}(\set{\emptyset})^{-1/2}\,\chfc{\set{\emptyset}}(Z)
\end{displaymath}
for $t\in[0,1]$. 
Then $\M^u=\M$.\label{prop:existencepsL1M}
\end{proposition}

All   examples of such measure types used in this paper  come from hitting sets of strong Markov processes $\sg tX$.  Basically, such sets are constructed by
\begin{displaymath}
  Z=\set{t+\tau:X_t=x^*}
\end{displaymath}
where $x^*$ is a suitable point and $\tau$ is a  random variable independent from $\sg tX$ with law equivalent to Lebesgue measure on $\NRp$. Please note that only almost sure properties of these random sets are important, not the special probabilistic structure. E.g., without loss of generality, we may assume $\tau\sim \mathrm{Exp_1}$.  

If $x^*$ is a suitable point then there is a nonnegative right-continuous increasing process $\sg sM$ with stationary independent increments upto a certain life time such that conditional on $X_0=x^*$,
\begin{displaymath}
  \set{t:X_t=x^*}=\overline{\set{M_s:s\ge0} }.
\end{displaymath}
$\sg sM$ is called \emph{subordinator}, see \cite{Ber02} for a thorough account on these processes and their range.

For a coarse classification of those random sets, remember the definition of Hausdorff-dimension of a set $Z$.  For $\alpha>0$ the $\alpha$-dimensional \emph{Hausdorff measure}  of a Borel set $Z$  is defined as 
\begin{equation}
  \label{eq:defHausdorff}
 H^\alpha(Z)=\sup_{\varepsilon>0}H^\alpha_\varepsilon(Z),
\end{equation}
where
  \begin{equation}
  \label{eq:defHausdorffa}
 H^\alpha_\varepsilon(Z)= \inf\set{\sum_{i\in\mathbb{N}}\Delta(B_i)^\alpha:\mbox{$\sequ iB$ are sets with  $\Delta(B_i)\le\varepsilon$ and $\bigcup_{i\in\mathbb{N}}B_i\supseteq Z$}},
  \end{equation}
denoting $\Delta(B)$ the diameter of $B$. Then  the \emph{Hausdorff dimension} $\dim_{H}Z$ of a Borel set $Z$ is  defined by 
\begin{displaymath}
  \dim_H Z=\inf\set{\alpha>0:H^{\alpha}(Z)>0  }.
\end{displaymath} 
We consider even more special sets, coming from Bessel diffusions: 
 \begin{example}[\cite{Tsi00}]\label{ex:bessel0}
Let $\sg t{X^{(d)}}$ be a Bessel diffusion  with parameter $d >0$ starting in a point $x_0>0$. This means $\sg t{X^{(d)}}$  is a strong Markov (diffusion) process on $\mathbb{R}_{\ge0}$ with generator
\begin{displaymath}
 \frac{\mathrm{d} \mathbb{E}_xf(X^{(d)}_t)}{\mathrm{d}t}\big\vert_{t=0}= \frac12f''(x)+\frac{d -1}{2x}f'(x).
\end{displaymath}
Throughout this work, $\mathbb{E}_x$ and $\mathbb{P}_x$ denote  the conditional expectation and conditional probability given $X_0=x$ respectively.
For $d \in \mathbb{N}$ we could realise this process via  $X^{(d)}_t=\norm{B^d_t}$, where $\sg t{B^d}$ is $d$-dimensional Brownian motion.
In the general case, the Bessel process   is also  defined as the (unique) nonnegative  solution of the stochastic differential equation  
\begin{displaymath}
\mathrm{d}X_t = \mathrm{d}W_t + \frac{d -1}{2}\frac 1{X_t}\mathrm{d}t. 
\end{displaymath}
Then we write $\sg t{X^{(d)}}\sim \mathrm{BES}(d,x_0)$.

According to the above mentioned scheme, define a random closed set $Z\in \C_{[0,1]}$ by 
\begin{displaymath}
  Z=\set{t\ge0:X^{(d)}_t=0}\cap[0,1]
\end{displaymath}
Observe that in this case the subordinator is stable of index $d $ \cite{Ber02}.
This means
\begin{displaymath}
  \mathbb{E}\mathrm{e}^{-\lambda M_s}=\mathrm{e}^{s\lambda^{d }}.
\end{displaymath}
Moreover,  for $d  \ge2$ $Z=\emptyset$ a.s. So we restrict to $d\in(0,2)$ for the rest of the paper.

Then the measure type 
$\M_d =\set{\mu:\mu\sim \mathscr{L}(Z)}$, which does not depend on $x_0$, is stationary factorising. Moreover, 
$\M_d $-a.s. the set $Z$ has  Hausdorff dimension $1-\frac d 2$ near every of its points. This means  for all $(s,t)$ with $Z\cap(s,t)\ne\emptyset$ it holds  $\dim_{H} (Z\cap(s,t))=1-\frac d2$.

As a consequence $Z$ has no isolated points: $\widehat Z=Z$. This immediately implies that the Arveson  system $\sg t\E$ determined by  $\E_t=L^2(\M_{0,t})$, $t\in[0,1]$, is type $\mathrm{II}_0$ \cite[Corollary 4.7]{Lie09a}, \cite{Tsi00}.   In the sequel, we denote this Arveson system by $\E^d$. Further, $\widehat Z=Z$ $\M_d-$a.s.\ also implies  $\M_d=\M^\U$. The latter measure type is an invariant of $\E$ by \cite[Theorem 3.22]{Lie09a} and we conclude that $\E^d\not\cong\E^{d'}$ for $d'\ne d$ (as long as both are $<2$), see also \cite{Tsi00}.

One more construction is useful in the sequel: The local time of the diffusion in 0. This local time, denoted $\sg tL$,  is the inverse of the subordinator $\sg sM$:
\begin{displaymath}
  L_t=\sup\set{s>0:\tau+M_s\le t}   
\end{displaymath}
Since $t\mapsto L_t$ is a random increasing nonnegative function, it is the cumulative distribution function of a random measure. It is easy to see that the support of this measure is just $Z$.  By results of \cite{FP71} this measure is just the restriction of a certain Hausdorff measure to $Z$. Thus this random measure depends on $Z$ only and we write $L_t(Z)$.
 \end{example}

\section{Bessel zeros yields primitive Arveson systems }

\begin{definition}
  A spatial Arveson system $(\E,u)$ is primitive, if $\mathscr{S}(\E)=\set{0,u,\E}$.

  A spatial Arveson system $(\E,u)$ is prime (spatially prime), if for Arveson systems $\F,\G$ with  $\F\otimes \G=\E$ ($\F\Utimes\G=\E$) it follows that either $\F$ or $\G$ is trivial, i.e.\ it equals $(\mathbb{C})_{t\ge0}$.  
\end{definition}
 According to \cite[Proposition 4.32, Note 4.33]{Lie09a} for all $k=1,2\dots$ there are uncountably many examples of type $\mathrm{II}_k$  Arveson systems which are prime and spatially prime. We now focus on examples of prime type $\mathrm{II}_0$ Arveson systems. 

In \cite{JMP11a} there was derived  a useful criterion for Arveson systems to be prime:

\begin{proposition}
  If for a spatial  Arveson  system $(\E,u)$  the lattice  $\mathscr{S}(\E)$ is totally ordered   then $\E$ is both prime and spatially prime. 

Especially, primitive Arveson systems  are both prime and spatially prime. 
\end{proposition}
\begin{proof}
  Analogous to \cite{JMP11a}. 
\end{proof}

The aim of the present section is the proof of
\begin{theorem}\label{cor:besseltrivial}
  $\E^d$ is primitive for all $0<d<2$. 
\end{theorem}
\begin{remark}
  This solves a question raised in \cite[Example 5.16]{Lie09a}. To our knowledge, these are  the first proven nontrivial examples of  primitive Arveson  systems. 
\end{remark}
For a proof, we still need some more structure. 

\begin{definition}
Suppose $\M$ is a stationary factorising measure type on $\mathcal{C}_{[0,1]}$. Then an $\M$-local stationary opening is a 
  measurable  map $\varphi:\mathcal{C}_{[0,1]}\mapsto\mathcal{C}_{[0,1]}$ with
\def\labelenumi{(\roman{enumi})}
\begin{enumerate}
\item
$\varphi(Z)\subseteq Z$ for all $ Z\in \mathcal{C}_{[0,1]}$,
\item $\varphi(Z+t)=\varphi(Z)+t$ for all $t\ge0$ and   $\M$-a.a. $Z$, and
\item 
\begin{displaymath}
  \varphi(Z\cap[s,t])=\varphi(Z)\cap[s,t]
\end{displaymath}
for all $0< s<t\le1$ for $\M$-a.a. $Z$.
\end{enumerate}
\end{definition}
\begin{remark}
  The name ``opening'' for  operators  with property (i) is common in  mathematical morphology, see e.g. \cite{Hei:94}. 
\end{remark}
The importance of this notion lies in
\def\citt{\cite[Lemma 5.14]{Lie09a}}
\begin{proposition}[\citt]\label{prop:openingsubsystem}
  Let  $\M$ be  a stationary factorising measure type on $\mathcal{C}_{[0,1]}$ and  $\E=\E^\M$ the associated Arveson system.  Suppose $\E$ is type $\mathrm{II}_0$ and  $\F\ne0$ is a subsystem of $\E$. 

Then there exists an $\M$-local stationary opening $\varphi$ such that
\begin{displaymath}
  \mathrm{P}^\F_{s,t}=\chfc{\set{Z:\varphi(Z)\cap[s,t]=\emptyset}}\dmf{0< s<t\le1}.
\end{displaymath}

Conversely, every $\M$-local stationary opening gives rise to a nonzero Arveson subsystem this way. 
\end{proposition}

The next  proposition is concerned with the probabilistic characterisation of Arveson subsystems of $\E^d$, or more generally Arveson systems arising from measure types of hitting sets of strong Markov processes. 
For a  stochastic process $\sg tX$ on a probability space $(\Omega,\Sigma,\mathbb{P})$ introduce the canonical (augmented) filtration $\Sigma^X$,
\begin{displaymath}
  \Sigma_t^X=\bigcap_{\varepsilon>0}\sigma(\set{X_s:s\le t+\varepsilon}\cup\set{B\in\Sigma:\mathbb{P}(B)=0})
\end{displaymath} 
 \begin{proposition}\label{prop:subrandomset}
   Let  $\sg tX$ be a strong Markov process with a.s.\ continuous paths in $\mathbb{R}^m$  such that for some $x^*\in \mathbb{R}^m$ the distribution of     
 \begin{displaymath}
  Z^X=\set {t\in[0,1]: X_t=x^*}\in \mathcal{C}_{[0,1]}
\end{displaymath}
is quasistationary and quasifactorising with measure type $\M$. 

If the filtration $\Sigma^X$ is right continuous then any $\M$-local stationary opening 
fulfils either $\varphi(Z)=\emptyset$ $\mathbb{P}$-a.s. or $\varphi(Z)=Z$ $\mathbb{P}$-a.s.
\end{proposition}
\begin{proof}
 For realisations with $Z^X=\emptyset$  there is nothing to prove.
  We introduce the random variable $\tau=\inf\set{t>0:X_t=x^*}$ such that $X_\tau=x^*$. Then the random variable
  \begin{displaymath}
    Y=\left\{
      \begin{array}[c]{cl}
        1&\tau\in \varphi(Z^X)\\
0&\tau\notin \varphi(Z^X)
      \end{array}
\right.
  \end{displaymath}
is well-defined. 
Now by the strong Markov property, the process $\sg t{\tilde X}$, $\tilde X_t=X_{\tau+t}$, is distributed according to $\mathbb{P}_{x^*}$. By definition and  locality of $\varphi$, $Y$ is $\bigcap_{\varepsilon>0}\Sigma^{\tilde X}_\varepsilon$-measurable.  Thus, by the Blumenthal 0-1 law, $\mathbb{P}(Y=1)\in\set{0,1}$. 
Moreover, from \cite[Theorem 22.13]{Kal01} we know that
\begin{displaymath}
  Z^X=\overline{\set{M_s+\tau:s\ge0}}\cap[0,1]
\end{displaymath}
where $\sg sM$ is the  subordinator associated with $X$ and $x^*$ which is  independent of $\tau$. It follows from the time symmetry of subordinators that we can apply the same arguments to the set $T-Z^X_{0,T}$. This means for $\tau_T=\sup\set{t<T:X_t=x^*}$ that $\mathbb{P}({\tau_T}\in \varphi(Z^X))\in\set{0,1}$, too.    
Introduce for all  $q\in \mathbb{Q}\cap\NRp$ random variables $Y_q^\pm\in\set{0,1}$:
\begin{displaymath}
Y_q^+=\left\{
\begin{array}[c]{cl}
1& \mbox{ if \quad$\inf (Z^X\cap(q,\infty))=\inf (\varphi(Z^X)\cap(q,\infty))$}\\ 0&\mbox{ otherwise}
\end{array}\right.
\end{displaymath}
and
\begin{displaymath}
Y_q^-=\left\{
\begin{array}[c]{cl}
1& \mbox{ if \quad$\sup (Z^X\cap(0,q))=\sup (\varphi(Z^X)\cap(0,q))$}\\ 0&\mbox{ otherwise}
\end{array}\right.
.
\end{displaymath}
 It is easy to see from  quasistationarity and quasifactorisation that there exists a  fixed $y\in \set{(0,0),(0,1), (1,0),(1,1)}$  such that it holds $\mathbb{P}$-a.s. $(Y_q^-,Y_q^+)=y $ for all positive $q\in \mathbb{Q}\cap\NRp$. 

It is clear that $y=(0,0)$ implies $\varphi(Z)=\emptyset$. Similarly,   $y=(1,1)$ implies $\varphi(Z)=Z$. 

Let us exclude $y=(0,1)$. Choose some $t\in Z^X\setminus\varphi(Z^X)$ and $q_n\nearrow_{n\to\infty}t$, $q_n\in \mathbb{Q}\cap (0,t)$. Then $Y^+_{q_n}=1$ indicates that there are $t_n\in \varphi(Z^X)$, $q_n<t_n<t$. This implies $\lim_{n\to\infty}t_n=t$. Since $\varphi(Z^X)$ is closed, $t\in \varphi(Z^X)$ contradicting  $t\in Z^X\setminus\varphi(Z^X)$. 

 The case $y=(1,0)$ is excluded by the same arguments.  This completes the proof.  
\end{proof}

\begin{remark}
  There is the more general bar code construction of \cite{Tsi03} giving a vast resource for examples of quasistationary quasifactorising random sets from hitting times sets of diffusions. Unfortunately,  \ref{prop:subrandomset} does not apply in general, for the hitting set is not point like in most situations.  
\end{remark}
\begin{proof}[ of \ref{cor:besseltrivial}]
   The claim follows from application of the previous result  and \ref{prop:openingsubsystem} to   $x^*=0$ and   $\sg tX\sim\mathrm{BES}(d,x_0)$.
\end{proof}

We can prove even more than primitivity for $\E^d$, its  gauge group is twodimensional. Remember the definition of $L_t(Z)$ from \ref{ex:bessel0} 
\begin{theorem}\label{th:autEdtrivial}
  Any $\theta\in \mathrm{Aut}(\E^d)$ has the form
  \begin{displaymath}
    \theta_tf(Z)=\mathrm{e}^{\mathrm{i}(\gamma_0 t+\gamma_1 L_t(Z))}f(Z)
  \end{displaymath}
for some real $\gamma_0,\gamma_1$. 
\end{theorem}
\begin{remark}
  A similar theorem holds for endomorphisms.
\end{remark}
\begin{proof}
  We know that $\theta$ should leave $\E^\U$ invariant. Thus there is $\gamma_0\in \mathbb{R}$ such that
  \begin{displaymath}
    \theta_t u_t=\mathrm{e}^{\mathrm{i}\gamma_0 t}u_t
  \end{displaymath}
for the standard unit of $\E^d$.  
Without loss of generality, let $\gamma_0=0$. 
Then $\theta_1$ shall commute with all the projections $\mathrm{P}^u_{s,t}$ defined by the unit through \ref{eq:defPst}. But those projections generate a maximal abelian von Neumann subalgebra of $\mathcal{B}(\E_1)$. Thus $\theta_1$ is in this subalgebra and we find a measurable function $\map \lambda{\C_{[0,1]}}{\mathbb{T}}$ such that 
\begin{math}
    \theta_1f(Z)=\lambda(Z)f(Z).
  \end{math}
Now we obtain from $\gamma_0=0$ that $\lambda(\emptyset)=1$. Furthermore,  for all $t$ and $\M_d-$a.a.~$Z\in \C_{[0,1]}$ it must hold
\begin{eqnarray*}
  \lambda(Z+t)&=&\lambda(Z)\\
\lambda (Z)&=&\lambda(Z_{0,t})\lambda(Z_{t,1}). 
\end{eqnarray*}
From these relations, we could extend $\lambda$ to $\bigcup_{n\ge1}\C_{[0,n]}$, e.g. 
\begin{displaymath}
\lambda(Z)=\lambda(Z_{0,1}) \lambda(Z_{1,2}-1). 
\end{displaymath}

Suppose now $Z\in \C_{\NRp}$ is the full zero set of a Bessel process with first hitting time $\tau $. 

Remember the definition of the subordinator $\sg sM$ from \ref{ex:bessel0}. Then it is easy to see from the strong Markov property and measurability of $\lambda$ that the $S^1$-valued process $\sg s\eta$,
\begin{displaymath}
  \eta_s(Z)=\lambda(Z\cap[0,\tau+M_s])
\end{displaymath}
has stationary independent multiplicative increments and measurable paths. Fix $\varepsilon>0$. Then the latter property shows that the set $S^\varepsilon(Z)$ of times $s$, where $\eta$ makes larger jumps  than $\varepsilon$, is locally finite almost surely. Consequently,
\begin{displaymath}
  \phi(Z)=\set{\tau+M_s:s\in S^\varepsilon (Z)}
\end{displaymath}
is an $\M_d$-local  stationary opening with $\phi(Z)\subsetneq Z$. By \ref{prop:openingsubsystem}, $\phi(Z)=\emptyset$ a.s. Since $\varepsilon$ was arbitrary,   $\eta$ must have  continuous paths a.s. 

As a consequence, there is almost surely a continuous version of
\begin{displaymath}
  t\mapsto \frac1{\mathrm{i}}\log \lambda(Z\cap[0,t])=\zeta_t(Z).
\end{displaymath}
Clearly, $\sg t\zeta$ is an additive functional of the Bessel process. Since $\lambda(\emptyset)=1$ $\zeta$ changes only on the zero set $Z$. By  \cite[Theorem 19.24]{Kal01}, $\zeta$ has to be a multiple of the local time. Thus there is some  $\gamma_1\in \mathbb{R}$ such that
\begin{displaymath}
  \lambda(Z)=\mathrm{e}^{\mathrm{i}\gamma_1L_1(Z)}
\end{displaymath}
for $\M_d-$a.a.\ $Z\in\C_{[0,1]}$.  As $\theta_1$ determines $\theta$, this completes the proof.   
\end{proof}

\section{Products of Arveson systems of  Bessel zeros }

Now we want to analise the spatial and tensor  products of the Arveson systems $\E^d,\E^{d'}$. 

First we want to check the condition from \ref{cor:alsoZhatdoesit}.
Remember that  $\dim_{H}(Z)$ is the Hausdorff dimension of any set $Z$. 
 \begin{theorem}
\label{thm:Kahaneps}
Assume that    $0<d _1,d _2<2$.

If $d _1+d _2\ge2$ then almost surely $Z_1\cap Z_2=\emptyset$ and  $ \E^{d_1}\Utimes\E^{d_2}=\E^{d_1}\otimes\E^{d_2}$.

If $d _1+d _2<2$ then with positive probability $Z_1\otimes Z_2\ne\emptyset$. Furthermore, almost surely  
for all $s<t$ with $Z_1\cap Z_2\cap (s,t)\ne \emptyset$
   \begin{equation}
\label{eq:Kahanepure}
       \dim_{H}(Z_1\cap Z_2\cap (s,t))= 1-\frac{d _1+d _2}2
   \end{equation}
Consequently, 
\begin{displaymath}
  \E^{d_1}\Utimes\E^{d_2}\subsetneqq\E^{d_1}\otimes\E^{d_2}
\end{displaymath}
then.
 \end{theorem}
 \begin{proof}
   By a result of \textsc{Shiga and Watanabe} \cite{SW73}, we know that for Bessel processes  $\sg tX\sim \mathrm{BES}(d ,x)$, $\sg t{X'}\sim \mathrm{BES}(d ',x')$ the process $Y$,
   \begin{displaymath}
     Y_t=\sqrt{X_t^2+(X'_t)^2}
   \end{displaymath}
is distributed as $\sg tY\sim \mathrm{BES}(d +d ',\sqrt{x^2+(x'^2)})$. Now 
\begin{displaymath}
  \set{t:X_t=0,X'_t=0}=\set{t:Y_t=0}
\end{displaymath}
Since we know that for any $0\le s<t$ $\mathcal{M}_{d+d'} $-a.s. either $Z\cap[s,t]=\emptyset$ or $\dim_{H}(Z\cap[s,t])=1-\frac{d+d'}2$
this proves the required statements. 

\ref{prop:alsoZhatdoesit} completes the proof.
 \end{proof}

 \begin{remark}
   Please note that we used the special structure of $\M_d $ here. Nevertheless, most of the implications hold true in much more generality, using techniques from   \cite{Kah86} to compute Hausdorff dimensions of stationary random sets.
 \end{remark}

\begin{theorem}
Suppose $d _1\ne d _2$, $0<d _1,d _2<2$.

Then
\begin{displaymath}
  \E^{d _1}\otimes\E^{d _2}\cong\E^{\M_{d _1}\ast\M_{d _2}}. 
\end{displaymath}

Moreover,   $\E^{d _1}\otimes \E^{d _2}$ has at most 5 proper  subsystems: $0$, $u\otimes v$, $u\otimes \E^{d _2}$, $\E^{d _1}\otimes v$, and  $\E^{d _1}\Utimes \E^{d _2}$. The last one appears  if and only if $d _1+d _2 <2$. Then it is not isomorphic to   $\E^{d _1}\otimes\E^{d _2}$.
\end{theorem}
\begin{proof}
Since the index of Arveson systems is additive,   $\E^{d _1}\otimes\E^{d _2}$ is of type $\mathrm{II}_0$ again. Thus it   has only one onedimensional subsystem $u\otimes v$. The measure type related to this embedding is the distribution of $Z_1\cup Z_2$ under $\M_{d_1}\otimes\M_{d_2}$. 

Assume w.l.o.g.\ $d _1>d _2$. Then we can almost surely recover $Z_2$ via
\begin{displaymath}
  Z_2=\set{t\in Z_1\cup Z_2: \dim_{H}((Z_1\cup Z_2)\cap (s,s'))=1-\frac{d _2}2 \forall s,s'\in \mathbb{Q}, s<t<s'}.
\end{displaymath}
Further,
 \ref{eq:Kahanepure} and $d _1> \frac{d_1+d _2}2-1$ show that $Z_1\setminus Z_2$ must be dense near every point of $Z_1$. This gives $Z_1=\overline{(Z_1\cup Z_2)\setminus Z_2}$.  

We conclude that the distribution of $Z_1\cup Z_2$ is  measure isomorphic to the distribution of $(Z_1,Z_2)$ or $\E^{d _1}\otimes\E^{d _2}\cong\E^{\M_{d _1}\ast\M_{d _2}}$.

Moreover, every  $\M_{d _1}\ast\M_{d _2}$-local stationary opening $\varphi$ induces an $\M_{d _1}$-local stationary opening $\varphi_1$ and  
an  $\M_{d _2}$-local stationary opening $\varphi_2$ if one of the two sets is empty. That means  for $\M_{d _1}-$a.a. $Z_1$ and $\M_{d _2}$-a.a. $Z_2$
\begin{displaymath}
  \varphi(Z_1\cup \emptyset)=\varphi_1(Z_1)\qquad\mbox{and}\qquad\varphi(\emptyset\cup Z_2)=\varphi_2(Z_2).
\end{displaymath}
By \ref{prop:openingsubsystem}, each of the maps  $\varphi_1$ and $\varphi_2$ is either  almost surely the identity or  almost surely constant to the empty set. 

If $\varphi_1(Z)=\emptyset$ for all $Z$, locality implies  $\varphi(Z_1\cup Z_2)\cap (s,t)=\emptyset$ for all $s,t$ such that $Z_2\cap(s,t)=\emptyset$. If additionally $\varphi_2(Z)=\emptyset$ for almost all $Z$, we see $\varphi(Z_1\cup Z_2)\subseteq Z_1\cap Z_2$. Now observe that $Z_1\cap Z_2$ is the set of zeros of $(X^{(d _1)}_t,X^{(d _2)}_t)_{t\ge0}$.  Applying \ref{prop:subrandomset}  again yields either $\varphi(Z_1\cup Z_2)=Z_1\cap Z_2$ or $\varphi(Z_1\cup Z_2)=\emptyset$ almost surely.  In the former case, we obtain the subsystem $\E^{d_1}\otimes\E^{d_2}$. In the latter case we find the subsystem $\E^{d_1}\Utimes\E^{d_2}$.

If for almost all $Z$ $\varphi_1(Z)=\emptyset$ and   $\varphi_2(Z)=Z$ then $\varphi(Z_1\cup Z_2)\cap (s,t)= Z_2\cap(s,t)$ if $Z_1\cap(s,t)=\emptyset$ such that $\varphi(Z_1\cup Z_2)=Z_2$. The subsystem must be $\E^{d_1}\otimes v$. 

Similar arguments work for $\varphi_2(Z)=\emptyset$ giving the subsystem $u\otimes\E^{d_2}$.

It remains the case $\varphi_1(Z)=\varphi_2(Z)=Z$. Then monotonicity implies $\varphi(Z_1\cup Z_2)\supseteq\varphi(Z_1\cup\emptyset)=Z_1$ and, equally, $\varphi(Z_1\cup Z_2)\supseteq Z_2$ such that $\varphi(Z)=Z$. The subsystem is $u\otimes v$.

Therefore, only the 5 listed subsystems are possible.  \ref{thm:Kahaneps} gives the assertion. 
\end{proof}
\begin{remark}
  This result is very similar to \cite[Theorem 3.5]{Pow04a}. But there only  the ``diagonal'' case $d _1=d _2$ is considered. The only formal difference we see is the use of all positive contractive cocycles as invariant, whereas we deal with  projection valued cocycles (corresponding to Arveson  subsystems). In our examples, the space of  nontrivial positive contractive cocycles of $\E^d$ is onedimensional, \cite{Pow04a} gives at least an estimate of dimension 2. This  indicates that the two examples are nonisomorphic. But, different from us, \cite{Pow04a} does not compute all subsystems.  Of course it would be quite interesting to translate the QP-flows used by  \cite{Pow04a} and others  into the random-set picture by computing their Arveson system.

To complete the picture a bit more, we present a somewhat surprising result in the diagonal case. 
\end{remark}

The ``diagonal case'' $d_1=d_2=d$ is more involved since we cannot transform the situation into a question involving one random set in $[0,1]$. We need direct integrals dealing with the multiplicity issue of representations of abelian von Neumann algebras, here $L^\infty(\M)$ for the  measure type $\M$ coming from embedding $\G=u\otimes u \subset\E$, see \cite[section 6]{Lie09a}.  
This theory  gives us
\begin{displaymath}
 \E_t= \int^\oplus\mu(\d Z)H^t_{Z}
\end{displaymath}
for a measurable family of Hilbert spaces $(H^t_Z)_{Z\in \C_{[0,1]}}$ and some $\mu\in \M_{0,t}$. But, also the change of measures and the product of the Arveson system should play a r\^ ole. 

For general embeddings $\G\subseteq\E$, we look at     a measurable family of Hilbert spaces $H=(H^t_Z)_{t\ge0, Z\in\C_{[0,t]}}$ with
\begin{enumerate}
\item $H^t_\emptyset=\G_t$ for all $t\ge0$.
\item There are unitaries $(V^{s,t}_{Z,Z'})_{s,t\ge0,Z\in\C_{[0,s]},Z'\in\C_{[0,t]}}$, $\map{V^{s,t}_{Z,Z'}}{H^s_{Z}\otimes H^t_{Z'}}{H^{s+t}_{Z\cup (Z'+s)}}$ which fulfil the associativity condition 
\begin{equation}
\label{eq:associativity VstZsZt}
 V_{Z,Z'\cup (Z''+s)}^{r,s+t}\circ(\unit_{H_{Z}^r}\otimes V^{s,t}_{Z',Z''})= V_{Z\cup( Z'+r),Z''}^{r+s,t}\circ(V_{Z\cup (Z'+r)}^{r,s}\otimes\unit_{H_{Z''}^t})
\end{equation}
for all $r,s,t\ge0$ and  for $\M_{0,r}$-a.a.\ $Z$, $\M_{0,s}$-a.a.\ $Z'$ and $\M_{0,t}$-a.a.\ $Z''$.
\end{enumerate}
Define a family $\F=\sg t\F$ of Hilbert spaces, 
\begin{displaymath}
    \F_t=\set{\psi=(\psi_\mu)_{\mu\in\M_{0,t}}:\psi_\mu\in \int^\oplus\mu(\d Z)H^t_{Z},\psi_{\mu'}=U_{\mu,\mu'}\psi_\mu\forall \mu,\mu'\in\M_{0,t} }.
  \end{displaymath}
Again, the unitaries $U_{\mu,\mu'}$ are given through \ref{eq:defUmumus}. 
$\F$  is equipped  with product unitaries $\sg {s,t}W$, $\map{W_{s,t}}{\F_s\otimes\F_t}{\F_{s+t}}$, given through
\begin{displaymath}
\label{eq:multiplication on EMH}
    (W_{s,t}\psi\otimes \tilde \psi')_{\mu_{0,s}\otimes(\mu'_{0,t}+s)}(Z\cup (Z'+s))=V^{s,t}_{Z,Z'}{\psi_{\mu_{0,s}}(Z)}\otimes{\psi'_{\mu'_{0,t}}(Z')}.
  \end{displaymath}
Then $\F$ is an Arveson system, see \cite[Lemma 6.6]{Lie09a},  denote it by $\E^{\M,H}$. 

We need the following result
\def\citt{\cite[Theorem 6.7]{Lie09a}}  
\begin{theorem}[\citt]
  \label{th:directintegralps}
  Let $\E$ be an Arveson  system,  $\G\subseteq\E$ a subsystem and $\M=\M^\G$ the corresponding measure type. 

Then there exists a measurable family of Hilbert spaces $H=(H^t_Z)_{t\ge0, Z\in\C_{[0,t]}}$ such that
   $\E\cong \E^{\M,H}$  under an isomorphism respecting the natural actions of $J_{\G}(L^\infty(\M^\G))$ and $L^\infty(\M^\G)$. 
\end{theorem}
 For the next result, let $ \mathbb{P}^1_{\mathbb{C}}$ denote the onedimensional complex projective space, i.e. the space of all onedimensional subspaces of $\mathbb{C}^2$.

\begin{theorem}
Suppose $d _1=d _2=d \ge1$.

Then $\M_{d }\ast\M_{d }=\M_{d }$ and thus $\E^d\otimes\E^d\not\cong\E^{\M_{d }\ast\M_{d }}$. 

Moreover,   $\E^d\otimes \E^d$ has infinitely many proper  subsystems: $0$, $u\otimes u$, and a continuum $(\E^z)_{z\in \mathbb{P}^1_{\mathbb{C}}}$ of subsystems isomorphic to $\E^d$. Thus, $\mathscr{S}(\E^d\otimes\E^d)$ has depth 4. 
\end{theorem}
\begin{proof}
  From   \ref{thm:Kahaneps} we know under $\M_d\otimes\M_d$ that $Z_1\cap Z_2=\emptyset$ a.s. Therefore,  \cite[Proposition 4.20]{Lie09a} shows $\M_d \ast\M_d =\M_d $.  This gives us another view on the random set $(Z_1,Z_2)\sim \M_d\otimes\M_d$ underlying the  Arveson system $\E^d\otimes\E^d$: we could condition on $Z=Z_1\cup Z_2$. If $\mathscr{L}(Z_1),\mathscr{L}(Z_2)=\mu$, we obtain the  conditional distribution of the pair $(Z_1,Z_2)$ given $Z_1\cup Z_2=Z$ as a stochastic kernel $q_\mu(\cdot|Z)$.  

Let us consider the direct integral representation of  $\E$. We derive by disintegration with respect to the measure $\mu\ast\mu\in\M_d$
\begin{displaymath}
  (\E^d\otimes\E^d)_t=\int_\oplus\mu\ast\mu(\mathrm{d} Z)H^t_{Z}
\end{displaymath}
with
\begin{displaymath}
  H^t_{Z}=L^2(q_\mu(\cdot|Z)).
\end{displaymath}
 Observe that for $\mu,\mu'\in \M_d$ the conditional distributions $q_\mu(\cdot|Z)$ and $q_{\mu'}(\cdot|Z)$ are equivalent for almost all $Z$.

Since  $Z_1\cap Z_2=\emptyset$ a.s., there is a partition $\mathfrak{t}=(t_1,\dots,t_n)\in \Pi^1$   such that
for all $i$ either $Z_1\cap[t_1+\dots+t_{i-1},t_1+\dots+t_{i}]=\emptyset$ or $Z_2\cap[t_1+\dots+t_{i-1},t_1+\dots+t_{i}]=\emptyset$. We describe this situation by $(Z_1,Z_2)\to\mathfrak{t}$.
  We could even choose the $t_i\in \mathbb{Q}$. Thus   there are only countably many choices of the partitions and pairs $(Z_1,Z_2)$ compatible with $Z_1\cup Z_2=Z$. We conclude that   $q_\mu(\cdot|Z)$ is a discrete measure. Further, $q_\mu(\set{(Z_1,Z_2)}|Z)=q_\mu(\set{(Z_2,Z_1)}|Z)$ since $\mu\otimes\mu$ is symmetric.

Now any Arveson  subsystem $\G$,  $u\otimes u\subset\G\subset \E$ is determined by Hilbert spaces $H'_{t,Z}\subseteq H_{t,Z}$ sharing the tensor products from the family $H$.  We introduce now the spaces 
\begin{displaymath}
  G_{t,Z}=\set{
    \begin{pmatrix}
      \psi(Z,\emptyset)\\\psi(\emptyset,Z)
    \end{pmatrix}
:\psi\in H'_{t,Z})}\subseteq \mathbb{C}^2,
\end{displaymath}
$t\in[0,1]$, $Z\in\C_{[0,t]}$.  By symmetry of $\mu\otimes\mu$, these spaces are independent from the choice of the measure $\mu\in \M_d$:
\begin{eqnarray*}
  U_{\mu\otimes\mu,\mu'\otimes\mu'}\psi(Z,\emptyset)&=&\sqrt{\frac{q_{\mu'}(\set{(Z,\emptyset)}|Z)}{q_{\mu}(\set{(Z,\emptyset)}|Z)}}\psi(Z,\emptyset)\\
  U_{\mu\otimes\mu,\mu'\otimes\mu'}\psi(\emptyset,Z)&=&\sqrt{\frac{q_{\mu'}(\set{(\emptyset,Z)}|Z)}{q_{\mu}(\set{(\emptyset,Z)}|Z)}}\psi(\emptyset,Z)=\sqrt{\frac{q_{\mu'}(\set{(Z,\emptyset)}|Z)}{q_{\mu}(\set{(Z,\emptyset)}|Z)}}\psi(\emptyset,Z).
\end{eqnarray*}

It is easy to see that the family $H'$ is uniquely determined by $G$. 
For, consider $Z$ distributed according to $\M_d$ and a partition $\mathfrak{t}\in \Pi^1$ like  mentioned above. Then
\begin{displaymath}
  H'_{1,Z}=H'_{t_1,Z_{0,t_1}}\otimes H'_{t_2,Z_{t_1,t_1+t_2}-t_1}\otimes\dots\otimes H'_{t_n,Z_{t_1+\dots+t_{n-1},t_1+\dots+t_n}-(t_1+\dots+t_{n-1})}
\end{displaymath}
So
\begin{displaymath}
  \set{(\psi(Z_1,Z_2))_{(Z_1,Z_2)\to\mathfrak{t}}: \psi \in H'_{1,Z}}\subseteq \mathbb{C}^{2^n},
\end{displaymath}
is fixed.  Varying $\mathfrak{t}$, we find that $H'_{1,Z}$ and consequently all $H'_{t,Z}$ are fixed by $G_{t,Z}$.

How does  $G_{t,Z}$  depend on $t$ and $Z$?      Of course, $G_{1,Z}=G_{1,Z+s} $ for all $s$, since $\G$ is a subsystem. 
Consider first
\begin{displaymath}
Q=\set{Z\in\C_{[0,1]}:
  \begin{pmatrix}
    \mathbb{C}\\1
  \end{pmatrix}
\cap G_{1,Z}\ne\emptyset}.
\end{displaymath}
 It is easy to see that $Z\in Q$ if and only if $ (Z_{0,t}\in Q)\wedge(Z_{t,1}\in Q)$. Also,  $Z\in Q$ if and only if $ Z+t\in Q$ $\M_d$-a.s. Thus $\chfc Q$ is the projection onto $\F_1$ for a subsystem of $\E^d$. By \ref{cor:besseltrivial} those are trivial and $\emptyset\in Q$, thus  either $Q=\set{\emptyset}$ or $Q=\C_{[0,1]}$ almost surely.    
A similar argument applies to $Q'=\set{Z\in\C_{[0,1]}:
  \begin{pmatrix}
    1\\\mathbb{C}
  \end{pmatrix}
\cap G_{1,Z}\ne\emptyset}$. 

If both $Q=Q'=\set{\emptyset}$, $G_{t,Z}=\set{
  \begin{pmatrix}
    0\\0
  \end{pmatrix}
}$ unless $Z=\emptyset$. This means $\G=u\otimes u$.

If $Q=\set{\emptyset}$, $Q'=\C_{[0,1]}$, $G_{t,Z}={
  \begin{pmatrix}
    0\\\mathbb{C}
  \end{pmatrix}
}$ unless $Z=\emptyset$. This means
\begin{displaymath}
  H'_{t,Z}=\set{\psi:\psi(Z_1,Z_2)=0\mbox{ unless }Z_1=\emptyset}
\end{displaymath}
The case $Q'=\set{\emptyset}$ and  $Q=\C_{[0,1]}$ is discussed similarly.

Now suppose  $Q=Q'=\C_{[0,1]}$, Let
\begin{displaymath}
Q''=\set{Z\in\C_{[0,1]}:\dim G_{1,Z}=1}.
\end{displaymath}
Again  we see that $Z\in Q''$ if and only if $ (Z_{0,t}\in Q'')\wedge(Z_{t,1}\in Q'')$. Furthermore,   $Z\in Q''$ if and only if $ Z+t\in Q''$ $\M_d$-a.s. So there are two possibilities: If $Q''=\set{\emptyset}$, $G_{t,Z}=\mathbb{C}^2$ unless $Z=\emptyset$. This means $H'_{t,Z}=H_{t,Z}$. Otherwise, $Q''=\C_{[0,1]}$ implies there is some $\lambda(t,Z)\in \mathbb{C}\setminus\set0$ such that $G_{t,Z}=\mathbb{C}
\begin{pmatrix}
  1\\\lambda(t,Z)
\end{pmatrix}
$.
Of course, $\lambda$ is stationary and fulfils almost surely 
\begin{displaymath}
  \lambda(1,Z)=\lambda(t,Z_{0,t})\lambda(1-t,Z_{t,1}-t)
\end{displaymath}
Since $\lambda$ has to be   measurable and $\lambda(t,\emptyset)=1$,  we find similar to \ref{th:autEdtrivial} some $w\in \mathbb{C}$  such that
\begin{displaymath}
  \lambda(t,Z)=\mathrm{e}^{w L_t(Z)}
\end{displaymath}
This completes the proof.  
\end{proof}
\begin{remark}
Clearly, the gauge group of $ \E^{d }\otimes\E^{d } $ is nontrivial. Nevertheless, this is already for $\E^{d }$ the case, see \ref{th:autEdtrivial}. Nevertheless, the gauge group of $ \E^{d }\otimes\E^d  $ is even not the direct square of the gauge groups. This resembles the type I case.  Loosely speaking, we would classify
\begin{center}
  \begin{tabular}[c]{ll}
    $\E^d$ &to be of type $\mathrm{II}_{\M_d,0}$
\\
    $\E^d\otimes\E^d$ &to be of type $\mathrm{II}_{\M_d,1}$
\\
    $\E^d\otimes\E^d\otimes \E^d$& to be of type $\mathrm{II}_{\M_d,2}$\\
&\vdots
  \end{tabular}
\end{center}
 We derive here some kind of conditional index. It is  given by $\dim G_{t,Z}-1$, which is essentially independent of $t$ and $Z$ as shown in the proof above. Even more, the sections $u^w_{t,Z}$, 
\begin{displaymath}
  u^w_{t,Z,\mu}(Z_1,Z_2)=c_{t,Z,\mu}\mathrm{e}^{wL_t(Z_1)}
\end{displaymath}
play the r\^ole of ``conditional units''.  For instance,  for carefully chosen constants $c_{t,Z,\mu}$ they fulfil
\begin{displaymath}
  u^w_{s,Z}\otimes u^w_{t,Z'}=u^w_{s+t,Z\cup (Z'+s)}.
\end{displaymath}
\end{remark}
\begin{remark}
This similarity with the Arveson system of  type $\mathrm{I}_1$ or  Arveson systems of type $\mathrm{II}_1$ as constructed from \cite[section 4.3]{Lie09a} gives rise to the following question: 
\begin{quote}
  Let  $\G_1,\G_2$ be two different but isomorphic subsystems of an Arveson system
  $\E$.

 Does there exist a $ \mathbb{P}^1_{\mathbb{C}}$-parametrised
  family of mutually different subsystems of $\E$, all of which are
  isomorphic to $\G_1,\G_2$?
\end{quote}
\end{remark}
\begin{remark}
  If $d <1$, we expect a more complicated structure and $\mathscr{S}(\E^d \otimes\E^d )$ to have again depth 5. Surely, there is a chain of length 5, but now we are not so sure about the subsystems ``between''  $\E^d\Utimes\E^d$ and $\E^d\otimes\E^d$.   At least there  seem to be parallels to  type $\mathrm{II}_1$ Arveson systems. 

Observe that the analysis of the \emph{spatial} product $\E^d\Utimes\E^d$ remains unchanged from the case $d\ge1$. 
\end{remark}

\end{document}